\documentclass[11pt,twoside,reqno,centertags,draft]{amsart}
\usepackage{amsfonts}
\usepackage{color,enumitem,graphicx}
\usepackage[colorlinks=true,urlcolor=blue,
citecolor=red,linkcolor=blue,linktocpage,pdfpagelabels,
bookmarksnumbered,bookmarksopen]{hyperref}

  \usepackage{amsmath,amsthm,amsfonts,amssymb}

\begin{document}
\title{ Classification of isolated singularities of nonnegative solutions to fractional semi-linear elliptic equations
and the existence results}
\date{}
\maketitle

\vspace{ -1\baselineskip}

{\small
\begin{center}

\medskip

  {\sc  Huyuan Chen\qquad Alexander Quaas}
 
\end{center}
}

\bigskip

\begin{quote}
{\bf Abstract.} In this paper, we classify the singularities of nonnegative solutions to fractional elliptic equation
\begin{equation}\label{eq 0.1}
  \arraycolsep=1pt
\begin{array}{lll}
 \displaystyle   (-\Delta)^\alpha    u=u^p\quad
 &{\rm in}\quad \Omega\setminus\{0\},\\[2mm]
 \phantom{ (-\Delta)^\alpha  }
 \displaystyle   u=0\quad
 &{\rm in}\quad \mathbb{R}^N\setminus\Omega,
\end{array}
\end{equation}
where  $p>1$,  $\Omega$ is a bounded, $C^2$ domain in $\mathbb{R}^N$ containing the origin, $N\ge2$ and
the fractional Laplacian $(-\Delta)^\alpha$ is defined in the principle value sense.
We obtain that  any classical solution $u$ of (\ref{eq 0.1}) is a weak solution of
\begin{equation}\label{eq 0.2}
  \arraycolsep=1pt
\begin{array}{lll}
 \displaystyle   (-\Delta)^\alpha    u=u^p+k\delta_0\quad
 &{\rm in}\quad \Omega,\\[2mm]
 \phantom{ (-\Delta)^\alpha  }
 \displaystyle   u=0\quad
 &{\rm in}\quad \mathbb{R}^N\setminus\Omega
\end{array}
\end{equation}
for some  $k\ge0$, where $\delta_0$ is the Dirac mass at the origin. In particular, when $p\ge \frac{N}{N-2\alpha}$, we have that $k=0$;
 when $p< \frac{N}{N-2\alpha}$, $u$ has removable singularity at the origin if $k=0$ and  if $k>0$,
$u$ satisfies
$$\lim_{x\to0} u(x)|x|^{N-2\alpha}=c_{N,\alpha}k,$$
where $c_{N,\alpha}>0$.

Furthermore,  when $p\in(1, \frac{N}{N-2\alpha})$,  we obtain that there exists  $k^*>0$ such that
 problem (\ref{eq 0.1}) has at least two positive solutions for $k<k^*$, a unique positive  solution for $k=k^*$
 and no positive solution for $k>k^*$.

\end{quote}

\renewcommand{\thefootnote}{}

\newcommand{\N}{\mathbb{N}}
\newcommand{\R}{\mathbb{R}}
\newcommand{\Z}{\mathbb{Z}}

\newcommand{\cA}{{\mathcal A}}
\newcommand{\cB}{{\mathcal B}}
\newcommand{\cC}{{\mathcal C}}
\newcommand{\cD}{{\mathcal D}}
\newcommand{\cE}{{\mathcal E}}
\newcommand{\cF}{{\mathcal F}}
\newcommand{\cG}{{\mathcal G}}
\newcommand{\cH}{{\mathcal H}}
\newcommand{\cI}{{\mathcal I}}
\newcommand{\cJ}{{\mathcal J}}
\newcommand{\cK}{{\mathcal K}}
\newcommand{\cL}{{\mathcal L}}
\newcommand{\cM}{{\mathcal M}}
\newcommand{\cN}{{\mathcal N}}
\newcommand{\cO}{{\mathcal O}}
\newcommand{\cP}{{\mathcal P}}
\newcommand{\cQ}{{\mathcal Q}}
\newcommand{\cR}{{\mathcal R}}
\newcommand{\cS}{{\mathcal S}}
\newcommand{\cT}{{\mathcal T}}
\newcommand{\cU}{{\mathcal U}}
\newcommand{\cV}{{\mathcal V}}
\newcommand{\cW}{{\mathcal W}}
\newcommand{\cX}{{\mathcal X}}
\newcommand{\cY}{{\mathcal Y}}
\newcommand{\cZ}{{\mathcal Z}}

\newcommand{\abs}[1]{\lvert#1\rvert}
\newcommand{\xabs}[1]{\left\lvert#1\right\rvert}
\newcommand{\norm}[1]{\lVert#1\rVert}

\newcommand{\loc}{\mathrm{loc}}
\newcommand{\p}{\partial}
\newcommand{\h}{\hskip 5mm}
\newcommand{\ti}{\widetilde}
\newcommand{\D}{\Delta}
\newcommand{\e}{\epsilon}
\newcommand{\bs}{\backslash}
\newcommand{\ep}{\emptyset}
\newcommand{\su}{\subset}
\newcommand{\ds}{\displaystyle}
\newcommand{\ld}{\lambda}
\newcommand{\vp}{\varphi}
\newcommand{\wpp}{W_0^{1,\ p}(\Omega)}
\newcommand{\ino}{\int_\Omega}
\newcommand{\bo}{\overline{\Omega}}
\newcommand{\ccc}{\cC_0^1(\bo)}
\newcommand{\iii}{\opint_{D_1}D_i}

\numberwithin{equation}{section}

\vskip 0.2cm \arraycolsep1.5pt
\newtheorem{lemma}{Lemma}[section]
\newtheorem{theorem}{Theorem}[section]
\newtheorem{definition}{Definition}[section]
\newtheorem{proposition}{Proposition}[section]
\newtheorem{remark}{Remark}[section]
\newtheorem{corollary}{Corollary}[section]

\setcounter{equation}{0}
\section{Introduction}

  Our   purpose of this paper is to classify the singularities of nonnegative solutions of fractional semi-linear elliptic problem
\begin{equation}\label{eq 1.1}
  \arraycolsep=1pt
\begin{array}{lll}
 \displaystyle   (-\Delta)^\alpha    u=u^p\quad
 &{\rm in}\quad \Omega\setminus\{0\},\\[2mm]
 \phantom{ (-\Delta)^\alpha  }
 \displaystyle   u=0\quad
 &{\rm in}\quad \R^N\setminus\Omega,
\end{array}
\end{equation}
where   $p>1$,  $\Omega$ is a $C^2$ bounded domain in $\mathbb{R} ^N$ containing the origin, $N\ge2$,
 the fractional Laplacian $(-\Delta)^\alpha$ is defined in the  principle value sense, i.e.
$$(-\Delta)^\alpha  u(x)=c_{N,\alpha}\lim_{\epsilon\to0^+} \int_{\R^N\setminus B_\epsilon(0) }\frac{ u(x)-
u(z)}{|x-z|^{N+2\alpha}}  dz.$$
Here $B_\epsilon(0)$ is the ball with radius $\epsilon$ centered at the origin and  $c_{N,\alpha}>0$ is the normalized constant, see \cite{NPV}.

When $\alpha=1$, $-\Delta$ is the well-known Laplace operator and  the related isolated singular problem
\begin{equation}\label{eq 1.4}
 \begin{array}{lll}
 \displaystyle
  -\Delta   u=u^p\quad  {\rm in}\quad \Omega\setminus\{0\}
  \\[2mm]
 \displaystyle
 u\ge0 \quad  {\rm in}\quad \Omega\setminus\{0\},\qquad
  u=0\quad
 {\rm on}\quad \partial\Omega
\end{array}
\end{equation}
 has been classified  by Lions in \cite{L} for $p\in (1,\frac{N}{N-2})$,
by  Aviles in \cite{A} for $p = \frac{N}{N-2}$, by Gidas and  Spruck in \cite{GS}
for $\frac{N}{N-2} < p < \frac{N+2}{N-2}$, by Caffarelli, Gidas and  Spruck in \cite{CGS} for
$p=\frac{N+2}{N-2}$. When $p\in (1,\frac{N}{N-2})$,  Lions in \cite{L} showed that any nonnegative  solution of (\ref{eq 1.4})
is a very weak solution of
  \begin{equation}\label{eq 1.5}
    \arraycolsep=1pt
\begin{array}{lll}
 \displaystyle    -\Delta    u=u^p+k\delta_0\quad
 &{\rm in}\quad \Omega,\\[2mm]
 \phantom{  -\Delta   }
 \displaystyle   u=0\quad
 &{\rm on}\quad  \partial\Omega
\end{array}
\end{equation}
for some $k\ge0$,   and further noted  that  there exists $k^*>0$ such that for $k\in(0,k^*)$, problem \eqref{eq 1.5} has at least two   solutions including the minimal solution and a Mountain Pass type solution; for $k=k^*$,  problem \eqref{eq 1.5} has a unique solution; there is no solution of  \eqref{eq 1.5} for $k>k^*$.
So the solution of (\ref{eq 1.4}) has either the singularity of $|x|^{2-N}$ or   removable singularity when
$p\in(1,\frac{N}{N-2})$. In contrast with   problem \eqref{eq 1.5} with source nonlinearity, V\'eron  in \cite{V0}
showed that   the semi-linear elliptic
equations with absorption terms
\begin{equation}\label{eq003}
  \arraycolsep=1pt
\begin{array}{lll}
 \displaystyle    -\Delta     u+u^p=0\quad
 &{\rm in}\quad \Omega\setminus\{0\},\\[2mm]
 \phantom{  -\Delta  +u^p }
 \displaystyle   u=0\quad
 &{\rm on}\quad  \partial\Omega,
\end{array}
\end{equation}
admits positive solutions, when $p\in(1,\frac{N}{N-2})$,  which satisfy
$${\rm  either}\ \, \lim_{x\to0}u(x)|x|^{N-2}=c_Nk\quad{\rm  or}\quad
\lim_{x\to0} u(x)|x|^{\frac{2}{p-1}}=c_p>0$$
for $k>0$, denoting by $u_k$ and $u_\infty$ respectively. Furthermore, $u_\infty$ is the limit of $\{u_k\}_k$ as $k\to+\infty$
and $u_k$ is a weak solution of
\begin{equation}\label{eq 1.6}
    \arraycolsep=1pt
\begin{array}{lll}
 \displaystyle    -\Delta    u+ u^p=k\delta_0\quad
 &{\rm in}\quad \Omega,\\[2mm]
 \phantom{  -\Delta + u^p  }
 \displaystyle   u=0\quad
 &{\rm on}\quad  \partial\Omega.
\end{array}
\end{equation}
 Such an object has been extended to  the equations with Radon measures or
boundary measure data  in   \cite{BB11,B12,GV,MV1,MV3} and more related topics see references \cite{BP2,BP,BV,V}.

When $\alpha\in(0,1)$, $(-\Delta)^\alpha$ is a non-local operator, which has been studied by  Caffarelli and Sivestre in \cite{CS1,CS2,CS3},    and  fractional equations with measures and absorption nonlinearity in type (\ref{eq 1.6}) have been studied by Chen and V\'eron in \cite{CV1,CV2}.  In the source nonlinearity case, Chen, Felmer and V\'eron in \cite{CFV} obtained one solution for
$$ \arraycolsep=1pt
\begin{array}{lll}
(-\Delta)^\alpha u=g(u)+\sigma\nu \quad & {\rm in}\quad\Omega,\\[2mm]
 \phantom{ (-\Delta)^\alpha  }
 u=0  \quad & {\rm in}\quad \R^N\setminus\Omega,
 \end{array}
$$
where $\sigma>0$ small, $\nu$ is a Radon measure and nonnegative function $g$ satisfies the integral subcritical condition
$$\int_1^{+\infty}g(s)s^{-1-\frac{N}{N-2\alpha}}\,ds<+\infty.$$

Our interest in this paper
is to classify the singularities of (\ref{eq 1.1}) and then to obtain the  existence of singular solutions of (\ref{eq 1.1}) by considering the  very weak solutions
of corresponding problem with Dirac mass.

The classification of singularities of nonnegative solutions for (\ref{eq 1.1}) states as follows.

\begin{theorem}\label{teo 0}
Assume that $p>1$ and $u$ is a nonnegative classical solution of (\ref{eq 1.1}).

     Then  $u\in L^p(\Omega)$ and
there exists $k\ge0$
  such that  $u$ is a very weak solution of
 \begin{equation}\label{eq 1.2}
    \arraycolsep=1pt
\begin{array}{lll}
 \displaystyle   (-\Delta)^\alpha    u=u^p+k\delta_0\quad
 &{\rm in}\quad \Omega,\\[2mm]
 \phantom{ (-\Delta)^\alpha  }
 \displaystyle   u=0\quad
 &{\rm in}\quad \R^N\setminus\Omega,
\end{array}
\end{equation}
that is, $u\in L^p$ and
 \begin{equation}\label{1.1}
     \int_\Omega [u (-\Delta)^\alpha \xi-u^p\xi]\, dx=k\xi(0),\quad \forall \xi\in C^\infty_c(\Omega),
\end{equation}
where $C^\infty_c(\Omega)$ is the space of all the functions in $C^\infty(\R^N)$ with the support in $\Omega$. Furthermore,

$(i)$\ \ When $p\ge \frac{N}{N-2\alpha}$, we have that $k=0$.

$(ii)$ When $p\in(1,\frac{N}{N-2\alpha})$,
  if $k=0$,   $u$ is a classical solution of
  \begin{equation}\label{eq 1.3}
  \arraycolsep=1pt
\begin{array}{lll}
 \displaystyle   (-\Delta)^\alpha    u=u^p\quad
 &{\rm in}\quad \Omega,\\[2mm]
 \phantom{ (-\Delta)^\alpha  }
 \displaystyle   u=0\quad
 &{\rm in}\quad \R^N\setminus\Omega;
\end{array}
\end{equation}
\qquad \quad if $k>0$, then $u$ satisfies
\begin{equation}\label{1.2}
 \lim_{x\to0} u(x)|x|^{N-2\alpha}=c_{N,\alpha} k.
\end{equation}

\end{theorem}

Notice that  for $\alpha=1$, by the local property of the Laplacian and Integration by Part formula,
the solution $u$ of (\ref{eq 1.4})  has the following essential estimate of the singularity of the average in sphere  of the corresponding solution
$$
 \bar u(r)\le \frac{c_1}{r^{N-2}},\qquad r>0\ \ {\rm small},
$$
where $c_1>0$ and $\bar u(r)=\int_{\partial B_r(0)} u(x)\, d\omega(x)$, then it is available  to apply the Schwartz's Theorem in \cite{S} to classify the singularity of solutions of
(\ref{eq 1.4}). However, for $\alpha\in(0,1)$, because of the nonlocal property of the fractional
Laplacian,  problem (\ref{eq 1.1}) can not be translated into ODE by the average sphere function.  The strategy to prove Theorem \ref{teo 0} is to derive $u\in L^p(\Omega)$ and to scale the typical test functions and by using the positiveness of the solution
to derive that
$$L(\xi):=\int_\Omega [u(-\Delta)^\alpha \xi-u^p\xi ]dx=k\xi(0),\quad \forall \xi\in C^\infty_c(\Omega).$$
We notice that $k=0$ in the super critical case, i.e. $p\ge \frac{N}{N-2\alpha}$, which means that the singularity of positive solution is not visible in the distribution sense.

From Theorem \ref{teo 0}, the  solution of (\ref{eq 1.1}) may have the singularity as $|x|^{2\alpha-N}$ or removable singularity at the origin.   Next we   consider the existence and nonexistence
singular solution of (\ref{eq 1.1}) by dealing with the very weak solutions to (\ref{eq 1.2}) when $p\in(1,\frac{N}{N-2\alpha})$.

 \begin{theorem}\label{teo 1}
Assume that $p\in(1,\frac{N}{N-2\alpha})$,  then there exists $k^*>0$,
  such that\\ $(i)$ for $ k\in(0, k^*)$, problem (\ref{eq 1.2}) admits a minimal positive solution $u_{k}$
and  a Mountain-Pass type solution $w_k>u_k$, both solutions are classical solutions of (\ref{eq 1.1}) and satisfy (\ref{1.2});\\
$(ii)$
for $k=k^*$, problem (\ref{eq 1.2}) admits a unique positive solution $u_{k}$ , which is a classical solution of (\ref{eq 1.1}) and satisfies (\ref{1.2});\\
$(iii)$
for $k\ge k^*$,  problem (\ref{eq 1.2}) admits no   solution.

\end{theorem}

We remark that the minimal positive solution of (\ref{eq 1.1}) is derived by iterating an increasing sequence  $\{v_n\}_n$ defined by
$$v_0= k \mathbb{G}_\alpha[\delta_0], \qquad v_n  = \mathbb{G}_\alpha[v_{n-1}^p]+ k \mathbb{G}_\alpha[\delta_0],$$
where  $\mathbb{G}_\alpha[\cdot]$ is the
Green operator defined as
$$\mathbb{G}_\alpha[f](x)=\int_{\Omega} G_\alpha(x,y)f(y)dy $$
and   $G_\alpha$ is  the Green kernel of $ (-\Delta)^\alpha $ in $\Omega\times\Omega$. The properties of Green's function see Theorem 1.1 in \cite{CS}.
To insure the convergence of the sequence  $\{v_n\}_n$, we need to construct a suitable barrier function by using the estimate
$$
 \mathbb{G}_\alpha[\mathbb{G}_\alpha^p[\delta_0]]\le c_2\mathbb{G}_\alpha[\delta_0]\quad {\rm in}\quad \Omega\setminus\{0\},
$$
where $c_2>0$.  By the analysis the stability of the minimal solution, we deduce the existence of the very weak solution in the case that
$k=k^*$ and for $k\in(0,k^*)$, we construct Mountain Pass solution $\upsilon_k$ for the problem
$$\arraycolsep=1pt
\begin{array}{lll}
 \displaystyle   (-\Delta)^\alpha   u= (u_{k}+u_+)^{p}- u_{k}^p \quad
 &{\rm in}\quad \Omega,\\[2mm]
 \phantom{  (-\Delta)^\alpha}
 u =0 &{\rm in}\quad \R^N\setminus\Omega
\end{array}
$$
and then the Mountain Pass type solution $\upsilon_k+u_k$ is a solution of (\ref{eq 1.2}).

The paper is organized as follows. In Section 2, we  show the integrability of the solution $u$ of (\ref{eq 1.1}) and the isoltated support of operator generated by $(-\Delta)^\alpha u-u^p$.   Section  3 is devoted
to do classification of the singularities of (\ref{eq 1.1}). Finally, in Section 4 we prove the existence and nonexistence of very weak solutions
 of problem (\ref{eq 1.2}).

 \section{Preliminary results }

 We start our analysis from the integrality of  nonnegative solution $u$ to the fractional problem (\ref{eq 1.1}).
 In what follows,  denote by $c_i$ the positive constant with $i\in\N$,
 $G_\alpha$ denotes the Green's function of $(-\Delta)^\alpha$ in $\Omega\times\Omega$ and
$\mathbb{G}_\alpha[\cdot]$ is the
Green operator defined as
$$\mathbb{G}_\alpha[f](x)=\int_{\Omega} G_\alpha(x,y)f(y)dy. $$

 \begin{proposition}\label{pr 1.1}
 Assume that $p>1$ and $u$ is a nonnegative classical solution of (\ref{eq 1.1}).
  Then
\begin{equation}\label{lp}
 u\in L^p(\Omega).
\end{equation}

 \end{proposition}
\begin{proof}
 If $u^p\not\in L^1(\Omega)$, then  it implies by $u\in L^\infty_{loc}(\Omega\setminus\{0\})$
that
$$\lim_{r\to0^+}\int_{\Omega\setminus B_r(0)} u^p\,dx=+\infty.$$
 So for any $r>0$,  there exist decreasing sequence $\{R_n\}_n$ such that $R_n\in(0,r)$,
$\lim_{n\to\infty}R_n=0$ and
\begin{equation}\label{2.0}
 \int_{B_r(0)\setminus B_{R_n}(0)} u^p\,dx=n.
\end{equation}
Let $v_n$ be the solution of
$$
  \arraycolsep=1pt
\begin{array}{lll}
 \displaystyle   (-\Delta)^\alpha   v_n=\chi_{\Omega\setminus B_{R_n}(0)}u^p\quad
 &{\rm in}\quad \Omega,\\[2mm]
 \phantom{ (-\Delta)^\alpha  }
 \displaystyle   v_n=0\quad
 &{\rm in}\quad \R^N\setminus\Omega,
\end{array}
$$
where $\chi_{O}=1$ in $O$ and $\chi_{O}=0$ in $\R^N\setminus O$ for any domain $O$ in $\R^N$.

Let $\Gamma_0$ be the Fundamental solution of
$$(-\Delta)^\alpha \Gamma_0=\delta_0\quad{\rm in}\quad  \R^N. $$
In fact, $\Gamma_0(x)=c_{N,\alpha} |x|^{2\alpha-N}$ for $x\in \R^N\setminus\{0\}$.
Since $u\ge 0$ in $\Omega\setminus\{0\},$
$\lim_{x\to0} (u+\Gamma_0)(x)=+\infty$
and $v_n$ is bounded in $\Omega$,
then there exists $r>0$ such that
$u+\Gamma_0\ge v_n$ in $B_r(0)\setminus\{0\}$ and
 it implies by Comparison Principle \cite[Theorem 2.3]{CFQ} that for any $n\in\N$
\begin{equation}\label{2.1}
u+\Gamma_0\ge v_n\quad{\rm in}\ \ \R^N\setminus\{0\}.
\end{equation}

Since $\lim_{y\to x}G(x,y)=+\infty$ for $x\in\Omega$, there exists $r_0>0$ such that $G(x,y)\ge 1$ for $x,y\in B_{r_0}(0)$, and by (\ref{2.0}),
\begin{eqnarray*}
v_n(x) &=& \mathbb{G}_\alpha[\chi_{\Omega\setminus B_{R_n}(0)} u^p]=\int_{\Omega\setminus B_{R_n}(0)} G(x,y) u^p(y)dy \\
  &\ge&  \int_{ B_{r_0}(0)\setminus B_{R_n}(0)} u^p (y)dy
  \\&=&n\to+\infty \quad\ {\rm as}\ \ \ n\to\infty,
\end{eqnarray*}
which together with (\ref{2.1}) implies that
$u+\Gamma_0=+\infty$ in $B_{r_0}(0)$ and this is impossible.
Therefore, we have that $u^p\in L^1(\Omega)$.
\end{proof}

\medskip

To improve the regularity,  we need following regularity result.
\begin{proposition}\label{embedding}\cite[Proposition 1.4]{RS}

Let $h\in L^s(\Omega)$ with $s\ge1$, then,
there exists $c_{3}>0$ such that

\noindent$(i)$
\begin{equation}\label{a 4.1}
\|\mathbb{G}_\alpha[h]\|_{L^\infty(\Omega)}\le c_{3}\|h\|_{L^s(\Omega)}\quad{\rm if}\quad \frac1s<\frac{2\alpha }N;
\end{equation}

\noindent$(ii)$
\begin{equation}\label{a 4.2}
\|\mathbb{G}_\alpha[h]\|_{L^r(\Omega)}\le c_{3}\|h\|_{L^s(\Omega)}\quad{\rm if}\quad \frac1s\le \frac1r+\frac{2\alpha}N\quad
\rm{and}\quad s>1;
\end{equation}

\noindent$(iii)$
\begin{equation}\label{a 4.02}
\|\mathbb{G}_\alpha[h]\|_{L^r(\Omega)}\le c_{3}\|h\|_{L^1(\Omega)}\quad{\rm if}\quad 1<\frac1r+\frac{2\alpha}N.
\end{equation}
\end{proposition}

In the searching the second solution of (\ref{eq 1.2}),  Mountain Pass theorem is applied  in the Hilbert space    $ H^{\alpha}_0(\Omega)$,  defined by the closure of $C^\infty_c(\Omega)$ under the norm of
$$\norm{v}_\alpha=\left(\int_{\R^N}\int_{\R^N}\frac{|v(x)-v(y)|^2}{|x-y|^{N+2\alpha}} \,dxdy \right)^{\frac12}.$$
The corresponding inner product in $ H^{\alpha}_0(\Omega)$ is given as
$$\langle u,v\rangle_\alpha=\int_{\R^N}\int_{\R^N}\frac{[u(x)-u(y)][v(x)-v(y)]}{|x-y|^{N+2\alpha}} \,dxdy.$$
\begin{proposition}\label{pr 2.2}
For $s\in[0,2\alpha)$, the embedding:
$$H^{\alpha}_0(\Omega) \hookrightarrow L^q(\Omega, |x|^{-s}dx)$$
is continuous and compact for
$$q\in [1,\, \frac{2N-s}{N-2\alpha}).$$
\end{proposition}
\begin{proof}
From Theorem 6.10 and Theorem 7.1 in \cite{NPV},  it is known that the embedding
$$H^{\alpha}_0(\Omega) \hookrightarrow L^q(\Omega)$$
is continuous for
$$q\in [1,\, \frac{2N}{N-2\alpha}]$$
and is compact for
$$q\in [1,\, \frac{2N}{N-2\alpha}).$$

By using H\"{o}lder inequality, for $q\in[1,\, \frac{2N-s}{N-2\alpha})$, let $t=\frac1q\frac{2N}{N-2\alpha}$, then
$$qt=\frac{2N}{N-2\alpha},\quad \quad \frac{st}{t-1}=\frac{s}{2N-q(N-2\alpha)}<N$$
and
\begin{equation}\label{2.02}
\int_\Omega \frac{u^q}{|x|^s}\, dx\le \left(\int_\Omega u^{qt}\, dx\right)^{\frac1t} \left(\int_\Omega |x|^{-\frac{st}{t-1}}\, dx\right) ^{1-\frac{1}{t}},
\end{equation}
thus, the embedding
$H^{\alpha}_0(\Omega) \hookrightarrow L^q(\Omega, |x|^{-s}dx)$
is continuous.
Now we choose $t_\epsilon=\frac1q\frac{2N}{N-2\alpha}-\epsilon$ for $\epsilon>0$ sufficient small,
then (\ref{2.02}) holds with $qt_\epsilon<\frac{2N}{N-2\alpha}$ and $\frac{st_\epsilon}{t_\epsilon-1}<N$, so
 the embedding
$H^{\alpha}_0(\Omega) \hookrightarrow L^q(\Omega, |x|^{-s}dx)$
is compact for
$q\in [1,\, \frac{2N-s}{N-2\alpha}).$
\end{proof}

\begin{lemma}\label{lm 4.1}
Let $\tau\in(0,N)$, then for $x\in B_{\frac12}(0)\setminus\{0\}$,
\begin{equation}\label{4.2.6}
\mathbb{G}_\alpha[|\cdot|^{-\tau}](x)\le
\left\{ \arraycolsep=1pt
\begin{array}{lll}
 c_2|x|^{-\tau+2\alpha} \quad
 &{\rm if}\quad \tau>2\alpha,\\[2mm]
 \displaystyle   -c_2\log (|x|) \quad
 &{\rm if}\quad \tau=2\alpha,\\[2mm]
 c_2  \quad
 &{\rm if}\quad \tau<2\alpha.
\end{array}
\right.
\end{equation}

For $p\in(1,\frac{N}{N-2\alpha})$,   there holds
\begin{equation}\label{4.2.1}
\mathbb{G}_\alpha[\mathbb{G}_\alpha^p[\delta_{0}]]\le
\left\{ \arraycolsep=1pt
\begin{array}{lll}
 c_2|x|^{p(2\alpha-N)+2\alpha} \quad
 &{\rm if}\quad p\in (\frac{2\alpha}{N-2\alpha},\frac{N}{N-2\alpha}),\\[2mm]
 \displaystyle   -c_2\log (|x|) \quad
 &{\rm if}\quad p=\frac{2\alpha}{N-2\alpha},\\[2mm]
 c_2  \quad
 &{\rm if}\quad p<\frac{2\alpha}{N-2\alpha}
\end{array}
\right.
\end{equation}
and
\begin{equation}\label{4.2.7}
 \mathbb{G}_\alpha[\mathbb{G}_\alpha^p[\delta_{0}]]\le  c_2\mathbb{G}_\alpha[\delta_{0}]\quad {\rm in}\quad \Omega\setminus\{0\}.
\end{equation}

\end{lemma}
\begin{proof}  For $x\in B_{\frac12}(0)\setminus\{0\}$, we have that
\begin{eqnarray*}
\mathbb{G}_\alpha[|\cdot|^{-\tau}](x)  &\le& c_{N,\alpha}  \int_{B_{R_0}(0)}\frac1{|x-y|^{N-2\alpha}}\frac1{|y|^{\tau}}dy \nonumber \\
    &=&   c_{N,\alpha}  |x|^{2\alpha-\tau } \int_{B_{\frac{R_0}{|x|}}(0)}\frac1{|e_{x}-y|^{N-2\alpha}}\frac1{|y|^{\tau}}dy
    \\&\le &c_4 |x|^{2\alpha-\tau }\int_{B_{\frac{R_0}{|x|}}(0)} \frac{1}{1+|y|^{ N-2\alpha+\tau}} dy
   \\&\le & \left\{ \arraycolsep=1pt
\begin{array}{lll}
 c_2|x|^{-\tau+2\alpha} \quad
 &{\rm if}\quad \tau>2\alpha,\\[2mm]
 \displaystyle   -c_2\log (|x|) \quad
 &{\rm if}\quad \tau=2\alpha,\\[2mm]
 c_2  \quad
 &{\rm if}\quad \tau<2\alpha,
\end{array}
\right.
\end{eqnarray*}
where $e_{x}=\frac{x}{|x|}$  and $R_0>0$ such that $\Omega\subset B_{R_0}(0)$.

From \cite{CS}, we know that
$$\mathbb{G}_\alpha[\delta_{0}](x)\le \frac{c_{N,\alpha}}{|x|^{N-2\alpha}}$$
 and
\begin{eqnarray*}
 \mathbb{G}_\alpha^p[\delta_{0}](x) \le   \frac{c_{N,\alpha}^p}{|x|^{(N-2\alpha)p}},\quad\forall x\in\Omega\setminus\{0\},
\end{eqnarray*}
then  we apply  (\ref{4.2.6}) to obtain (\ref{4.2.1}) and (\ref{4.2.7}).
\end{proof}

Let $\eta_0:\, \R^N\to [0,1]$ be a $C^\infty$ radially symmetric function increasing with respect to $|x|$ such that
$\eta_0=1$ in $\R^N\setminus B_2(0)$ and $\eta_0=0$ in $B_1(0)$. Let $\eta_\epsilon(x)=\eta_0(\epsilon^{-1} x)$ for  $x\in\R^N$
and
\begin{equation}\label{2.1-1}
 u_\epsilon=u\eta_\epsilon.
\end{equation}
By direct computation, we have that
\begin{eqnarray*}
 (-\Delta)^\alpha u_\epsilon (x) &=& \eta_\epsilon(x)(-\Delta)^\alpha u(x)+u(x)(-\Delta)^\alpha \eta_\epsilon(x)\\
    &&  -\frac{c_{N,\alpha}}{2}\int_{\R^N} \frac{(u(x)-u(y))(\eta_\epsilon(x)- \eta_\epsilon (y))}{|x-y|^{N+2\alpha}}dy,\quad\forall x\in \Omega\setminus\{0\}
\end{eqnarray*}
and
\begin{eqnarray*}
 (-\Delta)^\alpha u_\epsilon (0)=\lim_{x\to0} (-\Delta)^\alpha u_\epsilon (x)=-c_{N,\alpha}\int_{\R^N\setminus B_\epsilon(0)}  \frac{u(y)\eta_\epsilon(y)}{|y|^{N+2\alpha}}  dy.
\end{eqnarray*}

Denote  by $L$  the operator related to $(-\Delta)^\alpha u -u^p$ in the distribution sense,  i.e.
\begin{equation}\label{3.1}
L(\xi)=\int_\Omega [u(-\Delta)^\alpha \xi -u^p\xi]\,dx,\quad \forall\xi\in C^\infty_c(\Omega).
\end{equation}

\begin{lemma}\label{lm 2.0}
 For any $\xi\in C^\infty_c(\Omega)$ with
the support in $\Omega\setminus\{0\}$,
$$L(\xi)=0.$$

\end{lemma}
\begin{proof}
For any $\xi\in C^\infty_c(\Omega)$, applying the Integral by Parts formula, see Lemma 2.2 in \cite{CV1}, it infers that
 \begin{equation}\label{2.4}
  \int_\Omega \xi (-\Delta)^\alpha u_\epsilon \,dx=  \int_\Omega u_\epsilon  (-\Delta)^\alpha\xi\,dx.
 \end{equation}

Since $\xi\in C^\infty_c(\Omega)$ has the support in $\Omega\setminus\{0\}$, then there exists $r>0$ such that
$\xi=0$ in $B_r(0)$ and if we put the $\epsilon>0$ small enough, we have that
\begin{eqnarray*}
&&|\int_\Omega\left[ u (-\Delta)^\alpha \xi  -u^p\xi\right]\,dx|
\\&&  \le    |\int_\Omega\left[ u_\epsilon (-\Delta)^\alpha \xi  -u^p\xi\right]\,dx|+|\int_\Omega| (1-\eta_\epsilon) u (-\Delta)^\alpha \xi|\,dx
\\&&  \le  \int_{ B_{2\epsilon}(0)}|u (-\Delta)^\alpha \xi|\,dx+\int_\Omega  \left| \xi(-\Delta)^\alpha u_\epsilon  -u^p\xi \right|\,dx
\\&& \quad +\int_\Omega u |\xi||(-\Delta)^\alpha \eta_\epsilon |dx +\frac{c_{N,\alpha}}{2}\int_\Omega|\int_{\R^N} \frac{(u(x)-u(y))(\eta_\epsilon(x)- \eta_\epsilon (y))}{|x-y|^{N+2\alpha}}dy| |\xi(x)|\,dx
\\&&  =\int_{ B_{2\epsilon}(0)}|u (-\Delta)^\alpha \xi|\,dx+  \int_{\Omega\setminus B_r(0)}  \left| \xi(-\Delta)^\alpha u_\epsilon  -u^p\xi \right|\,dx +\int_{\Omega\setminus B_r(0)} u|\xi| |(-\Delta)^\alpha \eta_\epsilon(x)|dx \\&&\qquad+\frac{c_{N,\alpha}}{2}\int_{\Omega\setminus B_r(0)}|\int_{\R^N} \frac{(u(x)-u(y))(\eta_\epsilon(x)- \eta_\epsilon (y))}{|x-y|^{N+2\alpha}}dy||\xi(x)|\,dx,
\end{eqnarray*}
where $u_\epsilon$ is defined in (\ref{2.1-1}).
For $x\in \Omega\setminus B_r(0)$ and $\epsilon <\frac r4$, we have that
\begin{eqnarray*}
(-\Delta)^\alpha u_\epsilon (x) &=&c_{N,\alpha}{\rm P.V.}\int_{\R^N}\frac{u(x)-u_\epsilon(y)}{|x-y|^{N+2\alpha}}\, dy \\
    &=& (-\Delta)^\alpha u (x) +c_{N,\alpha}\int_{B_{2\epsilon}(0)}\frac{(1-\eta_\epsilon(y)) u(y)}{|x-y|^{N+2\alpha}}\, dy,
\end{eqnarray*}
where $|x-y|>r-2\epsilon$ and then we have that
$$
\lim_{\epsilon\to0} \norm{(-\Delta)^\alpha u_\epsilon -(-\Delta)^\alpha u}_{L^\infty(\Omega\setminus B_r(0))}=0,
$$
which implies that
\begin{equation}\label{2.2}
\lim_{\epsilon\to0} \int_{\Omega\setminus B_r(0)}\left|  \xi   (-\Delta)^\alpha u_\epsilon -u^p\xi \right|\,dx=0.
\end{equation}

For $x\in \Omega\setminus B_r(0)$ and $\epsilon <\frac r4$, we have that
\begin{eqnarray*}
|(-\Delta)^\alpha \eta_\epsilon (x)|  = c_{N,\alpha}\int_{B_{2\epsilon}(0)}\frac{1-\eta_\epsilon(y)}{|x-y|^{N+2\alpha}}\, dy \le    c \epsilon^N (r-2\epsilon)^{-N-2\alpha},
\end{eqnarray*}
then we imply that
\begin{equation}\label{2.3}
\lim_{\epsilon\to0}\int_{\Omega\setminus B_r(0)} u|\xi| |(-\Delta)^\alpha \eta_\epsilon(x)|dx =0.
\end{equation}

Finally, for $x\in \Omega\setminus B_r(0)$ and $\epsilon <\frac r4$, there holds that
\begin{eqnarray*}
 &&\int_{\Omega\setminus B_r(0)}|\int_{\R^N} \frac{(u(x)-u(y))(\eta_\epsilon(x)- \eta_\epsilon (y))}{|x-y|^{N+2\alpha}}dy| |\xi(x)|dx
 \\&&= \int_{\Omega\setminus B_r(0)}|\int_{B_{2\epsilon}(0)} \frac{(u(x)-u(y))(1- \eta_\epsilon (y))}{|x-y|^{N+2\alpha}}dy| |\xi(x)|dx \\
  &&\le   (r-2\epsilon)^{-N-2\alpha} \norm{\xi}_{L^\infty(\Omega)} \left[ 2^Nc_{N,\alpha} \epsilon^N \int_\Omega u(x)\,dx  +\int_{B_{2\epsilon}(0)} u(y) dy  \right]
  \\ &&\le   (r-2\epsilon)^{-N-2\alpha} \norm{\xi}_{L^\infty(\Omega)} \left[ 2^Nc_{N,\alpha} \epsilon^N \int_\Omega u(x)\,dx  +c_5\epsilon^{2\alpha}  \right]
  \\&&\to0\quad {\rm as}\quad \epsilon\to0,
\end{eqnarray*}
which, together with (\ref{2.2}) and (\ref{2.3}), implies that
$$\int_\Omega \left[u (-\Delta)^\alpha\xi   -u^p\xi\right]\,dx=0.$$
Therefore,
$L(\xi)=0$ for any $\xi\in C^\infty_c(\Omega)$ with
the support in $\Omega\setminus\{0\}$.
\end{proof}

\medskip

\section{Isolated singularities}

From  (\ref{lp}), then $u\in L^1(\Omega)$ and for any $\xi\in C^\infty_c (\Omega)$,
$$|\int_\Omega  u (-\Delta)^\alpha\xi\,dx| <+\infty,$$
 so $L $ is a bounded functional of $C^\infty_c (\Omega)$.
From Lemma \ref{lm 2.0},  for any $\xi\in C^\infty_c (\Omega)$ with
the support in $\Omega\setminus\{0\}$, then
$$L(\xi)=0.$$
This means that the support of $L$ is a isolated set $\{0\}$ and by Theorem XXXV in \cite{S} (see also Theorem 6.25 in \cite{R}) it implies that
\begin{equation}\label{S}
 L=\sum_{|a|=0}^\infty k_a D^{a}\delta_0,
\end{equation}
where $a=(a_1,\cdots,a_N)$ is a multiple index with $a_i\in\N$, $|a|= \sum_{i=1}^Na_i$ and in particular, $D^0\delta_0=\delta_0$.
Then we have that
\begin{equation}\label{3.3}
L(\xi)= \int_\Omega \left[u(-\Delta)^\alpha\xi -u^p\xi\right]\,dx=\sum_{|a|=0}^{\infty} k_a D^{a}\xi(0)\quad {\rm for}\quad \forall \xi\in C^\infty_c (\Omega).
\end{equation}

\begin{proposition}\label{pr 2.1}
Assume that $p>1$ and $a$ is multiple index.
Then
\begin{equation}\label{3.0}
 k_a=0\quad{\rm for\ any}\quad |a|\ge 2\alpha.
\end{equation}
In particular, if  $\alpha\in (0,\frac12]$, then there exists $k\ge0$ such that
$$L=k\delta_0.$$

\end{proposition}
\begin{proof}
For any multiple index $a=(a_1,\cdots,a_N)$, let $\zeta_a$ be a    $C^\infty$ function such that
\begin{equation}\label{3.2}
{\rm supp}(\zeta_a)\subset \overline{B_1(0)}\quad{\rm and}\quad \zeta_a(x)=k_{a} \prod_{i=1}^N x_i^{a_i} \quad {\rm for}\  \ x\in B_1(0).
\end{equation}
Now we use the following test functions in  (\ref{3.3}),
$$\xi_\epsilon(x):=\zeta_a(\epsilon^{-1}x),\quad \forall x\in\R^N.$$
Observe that
$$\sum_{|a|\le q} k_{a} D^{a}\xi_\epsilon(0)=\frac{k_{a}^2}{\epsilon^{|a|}} \prod^{N}_{i=1}a_i! ,$$
where  $ a_i!=a_i\cdot (a_i-1)\cdots1>0$  and $ a_i!=1$ if $a_i=0$.

Let $r>0$ and then
\begin{eqnarray*}
|\int_\Omega u(-\Delta)^\alpha\xi_{\epsilon}\,dx|  &=& \frac1{\epsilon^{2\alpha}} |\int_\Omega u(x)(-\Delta)^\alpha\zeta_a(\frac1\epsilon x)\,dx|   \\
    &\le & \frac1{\epsilon^{2\alpha}}  \left[\int_{\Omega\setminus B_r(0)} u(x)|(-\Delta)^\alpha\zeta_a(\frac1\epsilon x)|\, dx
    +\int_{B_r(0)} u(x)|(-\Delta)^\alpha\zeta_a(\frac1\epsilon x)|\, dx\right].
\end{eqnarray*}
Fix $r$, we see that
$$|(-\Delta)^\alpha\zeta_a(\frac1\epsilon x)|\to 0\quad{\rm as}\quad \epsilon\to0\quad {\rm uniformly\ in}\quad  \Omega\setminus B_r(0),$$
then
\begin{eqnarray*}
\int_{\Omega\setminus B_r(0)} u(x)|(-\Delta)^\alpha\zeta_a(\frac1\epsilon x)| dx &\to& 0 \quad{\rm as}\quad \epsilon\to0.
\end{eqnarray*}
Furthermore,
\begin{eqnarray*}
\int_{B_r(0)} u(x)|(-\Delta)^\alpha\zeta_a(\frac1\epsilon x)|\, dx &\le&\norm{(-\Delta)^\alpha\zeta_a}_{L^\infty(\R^N)} \left(\int_{B_r(0)} u^p(x)\,dx\right)^{\frac1p} |B_r(0)|^{\frac{p-1}{p}}  \\
    &\le& c_6\norm{(-\Delta)^\alpha\zeta_a}_{L^\infty(\R^N)}  \norm{u}_{L^p(\Omega)}   r^{\frac{p-1}{p}N}
    \\&\to& 0\quad {\rm as}\quad r\to0.
\end{eqnarray*}
Therefore, we have that
\begin{equation}\label{3.4}
|\int_\Omega u(-\Delta)^\alpha\xi_{\epsilon}\,dx|=\epsilon^{-2\alpha} o(1).
\end{equation}

On the other side, we see that
\begin{eqnarray*}
 \int_\Omega u^p \xi_{\epsilon}\,dx   &=&      \int_{ B_\epsilon(0)} u^p(x)\zeta_a(\frac1\epsilon x)\, dx
 \\&\le & \norm{\zeta_a}_{L^\infty(\Omega)}  \int_{B_\epsilon(0)} u^p(x) \, dx\to 0\quad {\rm as}\quad \epsilon\to0.
\end{eqnarray*}
For $|a|\ge 2\alpha$, we have that
$$k_{a}^2\le c_7\epsilon^{|a|}[\epsilon^{-2\alpha} o(1)+o(1)]\to 0\quad{\rm as}\quad \epsilon\to0,$$
then we have $k_{a}=0$ by arbitrary of $\epsilon$ in (\ref{3.2}) . Thus (\ref{3.0}) holds true.

In the particular of $\alpha\in(0,\frac12]$, since   $|a|\in \N$ and $2\alpha\le 1$, then we have that
  $k_a=0$ for all $|a|\ge 1$. The proof ends.
\end{proof}

From Proposition \ref{pr 2.1},  it implies that for $\alpha\in(\frac12,1)$, the expression (\ref{S}) reduces to
\begin{equation}\label{3.5}
 L=k\delta_0+\sum_{i=1}^N k_iD_i\delta_0.
\end{equation}
where $\langle D_i\delta_0,\xi\rangle=\frac{\partial \xi(0)}{\partial x_i}$.
We observe that
$$G_\alpha(x,y)=\frac{c_{N,\alpha}}{|x-y|^{N-2\alpha}}-g(x,y),$$
where $g$ is $\alpha-$harmonic function such that $g(x,y)=\frac{c_{N,\alpha}}{|x-y|^{N-2\alpha}}$
if $x\in\R^N\setminus\Omega$ or $y\in \R^N\setminus\Omega$,
 then we have that
 $$\mathbb{G}_\alpha[D_i\delta_0](x)=(N-2\alpha)\frac{x_i}{|x|^{N-2\alpha+2}}- \partial_{x_i} g(x,0) ,$$
and
$$|\partial_{x_i} g(x,0)|\le c_8\rho^{\alpha-1}(x),$$
where $c_8>0$ and $\rho(x)={\rm dist}(x,\R^N\setminus\Omega)$.

\begin{proposition}\label{pr 3.1}
Assume that $p>1$, $\alpha\in(\frac12,1)$, and $k_i\in \N$  is from (\ref{3.5}).
Then 
\begin{equation}\label{3.0.1}
 k_i=0\quad{\rm for\ any}\quad i=1,\cdots,N.
\end{equation}

\end{proposition}
\begin{proof}
Since
\begin{eqnarray*}
 \Gamma(x): &=& k \mathbb{G}_\alpha[\delta_{0}](x) +\sum_{i=1}^N k_i\mathbb{G}_\alpha[D_i\delta_0] \\
    &=& \frac{c_{N,\alpha}k}{|x|^{N-2\alpha}} +c_{N,\alpha}\sum_{i=1}^Nk_i\frac{x_i}{|x|^{N-2\alpha+2}}+c_{N,\alpha}\sum_{i=1}^N\partial_{x_i} g(x,0),
\end{eqnarray*}
where $g$ is a bounded function,
then $\Gamma$ must change signs if $k_i\not=0$ for some $i$.
Now assume that there exists $i $ such that $k_i\not=0$.
We observe that
\begin{equation}\label{3.2.1a}
u=\mathbb{G}_\alpha[u^p]+\Gamma
\end{equation}
and there is $t\in(0,1)$ such that
$$A_t:=\{x\in\Omega\setminus\{0\}:\, k_i\frac{x}{|x|}\cdot e_i>t\}\not=\emptyset$$
and
$$\frac1{c_9}|x|^{2\alpha-N-1}<\Gamma(x)\le c_9|x|^{2\alpha-N-1}+c_9\rho^{\alpha-1}(x),\quad x\in A_t,$$
where $c_9>1$.
So if $p\ge\frac{N}{N-2\alpha+1}$, then $\Gamma^p\not\in L^1( A_t)$ and $u^p\not\in L^1(\Omega)$,
which contradicts (\ref{lp}).

Now we only have to consider the case  that $p<\frac{N}{N-2\alpha+1}$. In order to obtain the contradiction,
we continue to estimate $\mathbb{G}_\alpha[u^p]$. Let
$$u_1= \mathbb{G}_\alpha[u^p ]. $$
We infer from $u^p\in L^{s_0}(\Omega)$ with $s_0=\frac{1}{2}[1+\frac1p\frac{N}{N-2\alpha+1}]>1$ and Proposition \ref{embedding} that
$u_1\in L^{s_1 p}(\Omega)$ and  $u_1^p\in L^{s_1}(\Omega)$ with
$$s_1=\frac1p\frac{N}{N-2\alpha s_0}s_0.$$
By \eqref{3.2.1a},
\begin{equation}\label{3.2.1b}
u^p\le c_{10}(u_1^p+ k^p|\Gamma|^p)\quad{\rm in}\ \ \Omega,
\end{equation}
where $c_{10}>0$. By the definition of $u_1$ and \eqref{3.2.1b}, we obtain
\begin{equation}\label{3.2.1c}
 u_1 \le c_{10}(\mathbb{G}_\alpha[u_1^p]+ k^p\mathbb{G}_\alpha[|\Gamma|^p]),
\end{equation}
where
$$ k^p\mathbb{G}_\alpha[ |\Gamma|^p](x)\le c_2|x|^{(2\alpha-1-N)p+2\alpha}$$
and
$$(2\alpha-1-N)p +2\alpha>2\alpha-1-N.$$
If $s_1>\frac1{2\alpha} Np$, by Proposition \ref{embedding},
$\mathbb{G}_\alpha[u_1^p]\in L^{\infty}(\Omega )$. Hence, we know from \eqref{3.2.1c} that
\begin{equation}\label{3.2.1d}
u_1(x)\leq c_{11}|x|^{(2\alpha-1-N)p+2\alpha},\quad\forall x\in\Omega\setminus\{0\}.
\end{equation}
In \eqref{3.2.1a},
$c_{N,\alpha}\sum_{i=1}^N k_i\frac{x_i}{|x|^{N-2\alpha+2}}$
has negative singularity as $|\cdot|^{2\alpha-1-N}$ in
$$-A_t:=\{x\in \Omega:\, -x\in A_t\},$$
then  from \eqref{3.2.1d} and  \eqref{3.2.1a} with $(2\alpha-1-N)p+2\alpha>2\alpha-1-N$,
 there exists some point $x_0\in -A_t$ such that
$u(x_0) <0,$
 which  is impossible since $u$ is nonnegative solution of (\ref{eq 1.1}).

On the other hand, if $s_1<\frac1{2\alpha} Np$, we proceed as above.
Let
$$u_2= \mathbb{G}_\alpha[ u_1^p]. $$
By Proposition \ref{embedding}, $u_2\in L^{s_2p}(\Omega)$, where
$$s_2=\frac1p\frac{Ns_1}{Np-2\alpha s_1 }>\frac{N}{N-s_0}s_1>\left(\frac1p\frac{N}{N-2\alpha s_0}\right)^2s_0.$$
Inductively, we define
$$s_m=\frac1p\frac{Ns_{m-1}}{Np-2\alpha s_{m-1}} >\left(\frac1p\frac{N}{N-2\alpha s_0}\right)^m s_0.$$
So there is $m_0\in\N$ such that
$$s_{m_0}>\frac{1}{2\alpha}Np$$
and by Proposition \ref{embedding} part $(i)$, it infers that
$$u_{m_0}\in L^\infty(\Omega).$$
Therefore, (\ref{3.2.1d}) holds true, it infers  that  $u(x_0)<0$ for some point $x_0\in \Omega\setminus\{0\}$ and we obtain a contradiction with  $u\ge0$.
\end{proof}

\medskip

\noindent{\bf Proof of Theorem \ref{teo 0}.} From Proposition \ref{pr 2.1} and Proposition \ref{pr 3.1}, there exists some $k\in\R$ such that
$u$ is a weak solution of
\begin{equation}\label{eq 3.1}
  \arraycolsep=1pt
\begin{array}{lll}
 \displaystyle   (-\Delta)^\alpha    u=u^p+k\delta_0\quad
 &{\rm in}\quad \Omega ,\\[2mm]
 \phantom{ (-\Delta)^\alpha  }
 \displaystyle   u=0\quad
 &{\rm in}\quad \R^N\setminus\Omega
\end{array}
\end{equation}
and then
$$u=\mathbb{G}_\alpha[u^p]+k\mathbb{G}_\alpha[\delta_0].$$
When $p\ge \frac{N}{N-2\alpha}$, if $k\not=0$,
$$u^p(x)\ge k^p\mathbb{G}_\alpha[\delta_0]^p(x)\ge c_{N,\alpha}^pk^p |x|^{-(N-2\alpha)p}\ge c_{N,\alpha}^pk^p|x|^{-N},\quad x\in \Omega\cap B_1(0)\setminus\{0\},$$
then $u^p\notin L^1(\Omega)$ and contradicts (\ref{lp}).
Therefore, if $p\ge \frac{N}{N-2\alpha}$, we have that $k=0$.

When $p\in (1,\frac{N}{N-2\alpha})$ and $k=0$,  then
$$u= \mathbb{G}_\alpha[u^p ]. $$
We infer from $u^p\in L^{t_0}(\Omega)$ with $t_0=\frac{1}{2}[1+\frac1p\frac{N}{N-2\alpha}]>1$ and Proposition \ref{embedding} that
$u\in L^{t_1 p}(\Omega)$ and  $u^p\in L^{t_1}(\Omega)$ with
$$t_1=\frac1p\frac{N}{N-2\alpha t_0}t_0.$$
If $t_1>\frac1{2\alpha} Np$, by Proposition \ref{embedding},
$u\in L^{\infty}(\Omega )$ and then it could be improved that $u$ is a classical solution of (\ref{eq 1.3}).

 If $t_1<\frac12 Np$, we proceed as above.
By Proposition \ref{embedding}, $u\in L^{t_2p}(\Omega)$, where
$$t_2=\frac1p\frac{Nt_1}{Np-2\alpha t_1 }>\frac{N}{N-t_0}t_1>\left(\frac1p\frac{N}{N-2\alpha t_0}\right)^2t_0.$$
Inductively, we define
$$t_m=\frac1p\frac{Nt_{m-1}}{Np-2\alpha t_{m-1}} >\left(\frac1p\frac{N}{N-2\alpha t_0}\right)^m t_0.$$
So there is $m_0\in\N$ such that
$$t_{m_0}>\frac{1}{2\alpha}Np$$
and by Proposition \ref{embedding} part $(i)$,
$$u\in L^\infty(\Omega),$$
 then it deduces that $u$ is a classical solution of (\ref{eq 1.3}).

 When $p\in (1,\frac{N}{N-2\alpha})$ and $k\not=0$, form observations that
$$ \lim_{x\to0} \mathbb{G}_\alpha[\delta_0](x)|x|^{N-2\alpha} =c_{N,\alpha}$$
and
\begin{equation}\label{13.2.1a}
u=\mathbb{G}_\alpha[u^p]+k\mathbb{G}_\alpha[\delta_0],
\end{equation}
we infer from $u^p\in L^{t_0}(\Omega)$ for any  $t_0\in (1,  \frac1p\frac{N}{N-2\alpha})$, letting
$$u_1= \mathbb{G}_\alpha[u^p ],$$
then it follows from Proposition \ref{embedding} that if $\frac1p\frac{N}{N-2\alpha}>\frac{Np}{2\alpha}$,
$u_1\in L^\infty(\Omega)$, then $u$ has asymptotic behavior $c_{N,\alpha}k|x|^{2\alpha-N}$ and we are done. If not,
$u_1\in L^{t_1 p}(\Omega)$ and  $u_1^p\in L^{t_1}(\Omega)$ with
$$t_1=\frac1p\frac{N}{N-2\alpha t_0}t_0.$$
By the Young's inequality,
\begin{equation}\label{13.2.1b}
u^p\le c_{11}\left(u_1^p+ |k|^p|x|^{p(2\alpha-N)}\right)\quad{\rm in}\ \ \Omega\setminus\{0\},
\end{equation}
where $c_{11}>0$. By the definition of $u_1$ and \eqref{13.2.1b}, we obtain
\begin{equation}\label{13.2.1c}
 u_1 \le c_{11}\left(\mathbb{G}_\alpha\left[u_1^p+ |k|^p|x|^{p(2\alpha-N)}\right]\right),
\end{equation}
where
$$ k^p\mathbb{G}_\alpha[ |\cdot|^{p(2\alpha-N)}](x)\le c_{12}|x|^{(2\alpha-N)p+2\alpha}$$
and
$$(2\alpha-N)p +2\alpha>2\alpha-N.$$
If $t_1>\frac1{2\alpha} Np$, by Proposition \ref{embedding},
$\mathbb{G}_\alpha[u_1^p]\in L^{\infty}(\Omega )$. Hence, we have that
\begin{equation}\label{13.2.1d}
u(x)\leq c_{13}\mathbb{G}_\alpha[u_1^p]+ c_{13}|x|^{(2\alpha -N)p+2\alpha}+k\mathbb{G}_\alpha[\delta_0](x),\quad\forall x\in\Omega\setminus\{0\}.
\end{equation}
Since $(2\alpha -N)p+2\alpha>2\alpha -N$, we deduce from  \eqref{13.2.1d} that
\begin{equation}\label{14.1}
\lim_{x\to0} u(x) |x|^{ N-2\alpha }=c_{N,\alpha} k.
\end{equation}

On the other hand, if $t_1<\frac1{2\alpha} Np$, we proceed as above.
Let
$$u_2= \mathbb{G}_\alpha[ u_1^p]. $$
By Proposition \ref{embedding}, $u_2\in L^{t_2p}(\Omega)$, where
$$t_2=\frac1p\frac{Ns_1}{Np-2\alpha t_1 }>\frac{N}{N-t_0}t_1>\left(\frac1p\frac{N}{N-2\alpha t_0}\right)^2t_0.$$
Inductively, we define
$$t_m=\frac1p\frac{Nt_{m-1}}{Np-2\alpha t_{m-1}} >\left(\frac1p\frac{N}{N-2\alpha t_0}\right)^m t_0.$$
So there is $m_0\in\N$ such that
$$t_{m_0}>\frac{1}{2\alpha}Np$$
and
$$u_{m_0}\in L^\infty(\Omega).$$
Therefore, by the assumption that $u$ is nonnegative, it is necessary that $k>0$ and
$$\lim_{x\to0}u(x)|x|^{N-2\alpha}=c_{N,\alpha}k.$$
This ends the proof. \hfill$\Box$

\section{Existence of weak solution}

\subsection{Minimal solution}

\noindent {\it Proof of Existence of the minimal solution in Theorem \ref{teo 1}.}   We first define the iterating sequence
$$v_0:= k \mathbb{G}_\alpha[\delta_{0}]>0,$$
and
$$
 v_n  =  \mathbb{G}_\alpha[v_{n-1}^p]+ k \mathbb{G}_\alpha[\delta_{0}].
$$
Observing that
$$v_1= \mathbb{G}_\alpha[(kv_0)^p] + k \mathbb{G}_\alpha[\delta_{0}]>v_0$$
and  assuming that
$$
v_{n-1} \ge  v_{n-2} \quad{\rm in} \quad \Omega\setminus\{0\},
$$
we deduce that
\begin{eqnarray*}
 v_n =   \mathbb{G}_\alpha[v_{n-1}^p]+ k \mathbb{G}_\alpha[\delta_{0}]
 \ge  \mathbb{G}_\alpha[v_{n-2}^p]+ k \mathbb{G}_\alpha[\delta_{0}]
 =   v_{n-1}.
\end{eqnarray*}
Thus, the sequence $\{v_n\}$ is a increasing with respect to $n$.
Moreover, we have that
\begin{equation}\label{4.2.3}
\int_{\Omega} v_n(-\Delta)^\alpha  \xi \,dx =\int_{\Omega} v_{n-1}^p\xi \,dx +k\xi(0), \quad \forall \xi\in C^\infty_c(\Omega).
\end{equation}

We next build an upper bound for the sequence $\{v_n\}$.  For  $t>0$, denote
\begin{equation}\label{wt}
 w_t=t k^p\mathbb{G}_\alpha[\mathbb{G}_\alpha^p[\delta_{0}]]+ k\mathbb{G}_\alpha[\delta_{0}]\le (c_2t k^p+ k)\mathbb{G}_\alpha[\delta_{0}],
\end{equation}
where $c_2>0$ is from Lemma \ref{lm 4.1},
 then
\begin{eqnarray*}
  \mathbb{G}_\alpha[w_t^p]+k\mathbb{G}_\alpha[\delta_{0}]\le  (c_2t k^p+ k)^p\mathbb{G}_\alpha[\mathbb{G}_\alpha^p[\delta_{0}]] +  k  \mathbb{G}_\alpha[\delta_{0}]
   \le  w_t,
\end{eqnarray*}
if
$$
 (c_2t k^p+ k)^p\le tk^p,
$$
that is
\begin{equation}\label{4.2.4}
  (c_2t k^{p-1} + 1)^p\le t.
\end{equation}

Note that the convex function $f_{k}(t) =  (c_2t k^{p-1} + 1)^p$  can intersect the line $g(t) = t$,  if
\begin{equation}\label{4.2.5}
    c_2 k^{p-1}\le \frac1p\left(\frac{p-1}p\right)^{p-1}.
\end{equation}
Let $k_p=\left(\frac1{c_2 p}\right)^{\frac1{p-1}}\frac{p-1}p$, then if $k\le k_p$, it always hold that $f_{k}(t_p)\le t_p$ for $ t_p=\left(\frac p{p-1}\right)^{p}.$
 Hence, for $t_p$ we have chosen,  by the definition of $w_{t_p}$, we have  $w_{t_p}>v_0$ and
$$v_1= \mathbb{G}_\alpha[v_0^p]+ k \mathbb{G}_\alpha[\delta_0]<\mathbb{G}_\alpha[w_{t_p}^p]+ k \mathbb{G}_\alpha[\delta_0]=w_{t_p}.$$
Inductively, we obtain
\begin{equation}\label{2.10a}
v_n\le w_{t_p}
\end{equation}
for all $n\in\N$. Therefore, the sequence $\{v_n\}$ converges. Let $u_{k}:=\lim_{n\to\infty} v_n$. By \eqref{4.2.3}, $u_{k}$ is a weak solution of (\ref{eq 1.2}).

We claim that $u_{k}$ is the minimal solution of (\ref{eq 1.1}), that is, for any positive solution $u$ of (\ref{eq 1.2}), we always have $u_{k}\leq u$. Indeed,  there holds
\[
 u  = \mathbb{G}_\alpha[  u^p]+ k \mathbb{G}_\alpha[\delta_0]\ge v_0,
\]
and then
\[
 u  = \mathbb{G}_\alpha[ u^p]+ k \mathbb{G}_\alpha[\delta_0]\ge \mathbb{G}_\alpha[ v_0^p]+ k \mathbb{G}_\alpha[\delta_0]=v_{1}.
\]
We may show inductively that
\[
u\ge v_n
\]
for all $n\in\N$.  The claim follows.

 Similarly,  if problem (\ref{eq 1.2}) has a nonnegative solution $u$  for $ k_1>0$, then (\ref{eq 1.2}) admits a minimal solution $u_{k}$ for all $ k\in(0, k_1]$. As a result, the mapping $ k\mapsto u_{k}$ is increasing.
So we may define
$$k^*=\sup\{k>0:\ (\ref{eq 1.1})\ {\rm has\ minimal\ solution\ for\ }k \},$$
then $k^*$ is the largest $k$ such that problem (\ref{eq 1.2}) has minimal positive solution, and
$$k^*\ge k_p.$$

We next prove that $\lambda^*<+\infty$.  Let $(\lambda_1,\varphi_1)$ be the first eigenvalue and positive eigenfunction of $(-\Delta)^\alpha$ in $H^\alpha_0(\Omega)$,
see Proposition 5 in \cite{SV}
then by the fact that
$$u_k\ge k\mathbb{G}_\alpha[\delta_0]\ge c_{14}k \quad{\rm in}\quad B_r(0)$$
for some $r>0$ satisfying $B_{2r}(0)\subset \Omega$. There exists $c_{14}>1$ such that
$$c_{14}\int_{B_r(0)} u_{k} \varphi_1\,dx\ge \int_{\Omega} u_{k} \varphi_1\,dx$$
and we imply that
\begin{eqnarray*}
c_{14}\lambda_1\int_{B_r(0)} u_{k} \varphi_1\,dx &\ge& \lambda_1 \int_{\Omega} u_{k} \varphi_1\,dx= \int_{\Omega} u_{k}(-\Delta)^\alpha \varphi_1\,dx\\& \ge  &\int_{\Omega}   u_{k}^p\varphi_1 \,dx+k\varphi_1(x_0)
 \ge  c_{15}  k^{p-1} \int_{B_r(0)} u_{k} \varphi_1\,dx,
\end{eqnarray*}
then $k$ must satisfy
$$
 k\le c_{16}\lambda_1^{-\frac1{p-1}} .
$$
where $c_{15},\, c_{16}>0$.
As a conclusion,  there exists $c_{17}>0$ such that
\begin{equation}\label{2.14}
 k^*\le c_{17} \lambda_1^{-\frac1{p-1}}.
\end{equation}

{\it Regularity of very weak of solution of (\ref{eq 1.2}). }  Let $u$ be a very weak solution of (\ref{eq 1.2}) and $x_0\in \Omega\setminus\{0\}$, then
\begin{eqnarray*}
u  &=& \mathbb{G}_\alpha[u^p]+k\mathbb{G}_\alpha[\delta_0]  \\
   &=&  \mathbb{G}_\alpha[u^p\chi_{B_r(x_0)}]+\mathbb{G}_\alpha[u^p\chi_{\Omega\setminus B_r(x_0)}] +k\mathbb{G}_\alpha[\delta_0]
\end{eqnarray*}
where $\mathbb{G}_\alpha[\delta_0]$ is $C^\infty_{loc}(\Omega\setminus\{0\})$, $r>0$ such that $ \overline{B_{2r}(0)}\subset \Omega\setminus\{0\}$.
 Let $B_i=B_{2^{-i}r}(x_0)$.
For $x\in B_{i}$, we have that
$$\mathbb{G}_\alpha[\chi_{\Omega\setminus B_{i-1}} u^p](x)=\int_{\Omega\setminus B_{i-1}}   u^p(y)G_\alpha(x,y)dy,$$
then, for some $C_i>0$, we have
\begin{equation}\label{3.01-1}
\norm{\mathbb{G}_\alpha[\chi_{\Omega\setminus B_{i}}  u^p]}_{C^2(B_{i-1})}\le C_i \norm{ u^p}_{L^1(B_{2r}(x_0))}
\end{equation}
and for some   constant $c_i>0$ depending on $i$, we have \begin{equation}\label{3.01}
\norm{\mathbb{G}_\alpha[\delta_0]}_{C^2(B_{i-1})} \le   c_i|x_0|^{2-N}.
\end{equation}
By  Proposition \ref{embedding}, $u^p\in L^{q_0}(B_{2r_0}(x_0))$ with $q_0=\frac12(1+\frac1p\frac{N}{N-2\alpha})>1$.
By Proposition \ref{embedding} again we find
$$\mathbb{G}_\alpha[\chi_{B_{2r}(x_0)} u^p]\in L^{p_1}(B_{2r}(x_0))\ {\rm with }\ \ p_1=\frac{Nq_0}{N-2\alpha q_0}.$$
Similarly,
$$ u^p\in L^{q_1}(B_{r}(x_0))\ {\rm with }\ \   q_1=\frac{p_1}{p},$$
and
$$\mathbb{G}_\alpha[\chi_{B_{r}(x_0)}u^p]\in L^{p_2}(B_{r}(x_0))\ {\rm with }\ \ p_2=\frac{Nq_1}{N-2\alpha q_1}.$$
Let $q_i=\frac{p_{i}}{p}$ and $p_{i+1}=\frac{Nq_i}{N-2q_i}$ if $N-2q_i>0$. Then we obtain inductively that
$$ u^p\in L^{q_i}(B_{i}) \quad\mbox{and}
\quad \mathbb{G}[\chi_{B_{i}} u^p]\in L^{p_{i+1}}(B_{i}).$$
We may verify that
$$\frac{q_{i+1}}{q_i}=\frac1p\frac N{N-2\alpha q_i}>\frac1p\frac{N}{N-2\alpha q_1}>1.$$
Therefore,
$\lim_{i\to+\infty} q_i=+\infty,$
so there exists $i_0$ such that $N-2q_{i_0}>0$, but $N-2q_{i_0+1}<0$, and we deduce that
$$\mathbb{G}_\alpha[\chi_{B_{i_0}} u^p]\in L^{\infty}(B_{{i_0}}).$$
As a result,
$$ u\in L^{\infty}(B_{i_0}).$$
By elliptic regularity, we know from (\ref{3.01}) that $u$ is H\"{o}lder continuous in $B_{{i_0}}$ and  so is $ u^p$.  Hence, $u$ is a classical solution of  (\ref{eq 1.1}).
  \hfill$\Box$

\subsection{Stability}

In what follows, we discuss the stability of the minimal solution of (\ref{eq 1.1}).

\begin{definition}\label{def 1}
A solution (or weak solution) $u$ of (\ref{eq 1.1}) is stable (resp. semi-stable) if
$$
\norm{\xi}_\alpha^2> p\int_{\Omega}  u^{p-1}\xi^2 \,dx,\quad ({\rm resp.}\ \ge)\quad \forall \xi\in H^{\alpha}_0(\Omega)\setminus\{0\}.
$$

\end{definition}

\begin{proposition}\label{pr 3.2}

For $k\in (0,k^*)$, let $u_{k}$ be the minimal positive solution of (\ref{eq 1.1}) by Thoerem \ref{teo 0}.
Then $u_{k}$ is stable.

Moreover, there exists $c_{18}>0$ such that for any $\xi\in H^\alpha_0(\Omega)\setminus\{0\},$
\begin{equation}\label{e.3.1}
\norm{\xi}_\alpha^2- p\int_{\Omega} u_{k}^{p-1}\xi^2 \,dx\ge c_{18}\left((k^*)^{\frac{p-1}p}-k^{\frac{p-1}p}\right)\norm{\xi}_\alpha^2.
\end{equation}

\end{proposition}
\begin{proof}  {\it   To prove the  stability when $k>0$ small.}
When $k>0$ small, the iteration  procedure: $v_n=\mathbb{G}_\alpha[v_{n-1}^p]+ k\mathbb{G}_\alpha[\delta_{0}]$ is controlled by  super solution $w_{t_p}$,
where
$$ w_{t_p}=t_p k^p\mathbb{G}_\alpha[\mathbb{G}_\alpha^p[\delta_{0}]]+ k\mathbb{G}_\alpha[\delta_{0}],$$
then  $u_k\le w_{t_p}$ and  there exists $c_{19}>0$   such that
$$u_{k}(x)\le c_{19}k|x|^{2\alpha-N},\qquad\forall x\in \Omega\setminus\{0\}.$$
So there holds
$$ u_{k}^{p-1}(x)\le c_{20}k^{p-1} |x|^{(2\alpha-N)(p-1)}.$$
Then it follows by Proposition \ref{pr 2.2} that
\begin{eqnarray*}
\int_{\Omega} u_{k}^{p-1}\xi^2 \,dx  \le    c_{20} k^{p-1}\int_{\Omega} \frac{\xi_n^2(x)}{|x|^{(N-2\alpha)(p-1)}} \,dx
   <   \frac1p\norm{\xi}^2_\alpha
\end{eqnarray*}
if $k>0$ sufficient small.
Then $u_{k} $ is a stable solution of (\ref{eq 1.1}) for $k>0$ small.\smallskip

\emph{Now we prove the stability for  $k\in(0,k^*)$.}
Suppose that if $u_k$ is not stable, then we have that
\begin{equation}\label{e.3.5}
\sigma_1:= \inf_{\xi\in H^\alpha_0(\Omega)\setminus\{0\}}\frac{\norm{\xi}_\alpha^2}{p\int_{\Omega}  u_k^{p-1}\xi^2 \,dx}\le 1.
\end{equation}
 By Proposition \ref{pr 2.2}, $\sigma_1$ is achievable and  its achieved function $\xi_1$ could be setting to be nonnegative and   satisfies
$$(-\Delta)^\alpha    \xi_1 =  \sigma_1 pu_{k}^{p-1}   \xi_1.$$
Choosing $\hat{k}\in (k,k^*)$ and letting $w=u_{\hat k } -u_{k }>0$, then we have that
$$w=\mathbb{G}_\alpha[u_{\hat k}^p- u_{k}^p]+(\hat{k}-k)\mathbb{G}_\alpha[\delta_0].$$
By the elementary inequality
$$(a+b)^p\ge a^p+pa^{p-1}b \quad{\rm for}\ \ a, b\ge 0,$$
we infers that
$$w\ge  \mathbb{G}_\alpha[p u_{k}^{p-1} w]+(\hat{k}-k)\mathbb{G}_\alpha[\delta_0].$$
Then
\begin{eqnarray*}
\sigma_1\int_{\Omega}pu_{k}^{p-1} w\xi_1\,dx    &= &  \int_{\Omega} \xi_1(-\Delta)^\alpha w\,dx
 \\ &\ge & \int_{\Omega}p u_{k}^{p-1} w\xi_1\,dx+(\hat{k}-k)\xi_1(0)>   \int_{\Omega}p u_{k}^{p-1} w\xi_1\,dx,
\end{eqnarray*}
which is impossible. Consequently,
$$
 p\int_{\Omega}    u_{k}^{p-1}  \xi^2 \,dx> \norm{\xi}_\alpha^2,\qquad \forall \xi\in H^\alpha_0(\Omega).
$$
As a conclusion, we derive that $u_{k}$ is stable for $k<k^*$.

{\it To prove (\ref{e.3.1}).} For any $k\in (0,k^*),$ let $k'=\frac{k+k^*}{2}>k$ and $l_0=(\frac k{k'})^{\frac1{p}}<1$, then
we see that the minimal solution $u_{k'}$ of (\ref{eq 1.1}) with $k'$ is   stable and
\begin{eqnarray*}
 l_0u_{k'}&\ge&l_0^p u_{k'}
 \\&=&l_0^p\left(\mathbb{G}_\alpha[u_{k'}^p]+k'\mathbb{G}_\alpha[\delta_0]\right)+(k-k'l_0^p)\left\{\mathbb{G}_\alpha[\frac{u_{k'}}{|x|^2}]+\mathbb{G}_\alpha[\delta_0]\right\}
 \\&=& \mathbb{G}_\alpha[(l_0u_{k'})^p]+\mathbb{G}_\alpha[\frac{l_0u_{k'}}{|x|^2}]+k\mathbb{G}_\alpha[\delta_0],
\end{eqnarray*}
where we have used $k-k'l_0^p=0$. Thus, we have that $l_0u_{k'}$ is the minimal solution of (\ref{eq 1.1}) and
we have that
$$l_0u_{k'}\ge u_{k},$$
so for $\xi\in  H^\alpha_0(\Omega)\setminus\{0\}$, we have that
\begin{eqnarray*}
0< \norm{\xi}_\alpha^2- p\int_{\Omega}   u_{k'}^{p-1}\xi^2 \,dx
& \le&   \norm{\xi}_\alpha^2- pl_0^{1-p}\int_{\Omega}   u_{k}^{p-1}\xi^2 \,dx
\\&=&l_0^{1-p}\left[l_0^{p-1}\norm{\xi}_\alpha^2- p\int_{\Omega}  u_{k}^{p-1}\xi^2 \,dx\right],
\end{eqnarray*}
thus,
\begin{eqnarray*}
 \norm{\xi}_\alpha^2- p\int_{\Omega} u_{k}^{p-1}\xi^2 \,dx
 &=& (1-l_0^{p-1})\norm{\xi}_\alpha^2+  \left[l_0^{p-1}\norm{\xi}_\alpha^2- p\int_{\Omega} u_{k}^{p-1}\xi^2 \,dx\right]
   \\ & \ge& (1-l_0^{p-1}) \norm{\xi}_\alpha^2,
\end{eqnarray*}
which together with the fact that
 $$ 1-l_0^{p-1}\ge c_{21}[(k^*)^{\frac{p-1}p}-k^{\frac{p-1}p}], $$
 implies (\ref{e.3.1}).
\end{proof}

\subsection{Extremal solution } We would like to approach the weak solution when $k=k^*$ by the minimal solution $u_k$ with $k<k^*$.

\smallskip

\noindent{\it Proof of Theorem \ref{teo 1} in the case of $k=k^*$. }
Let $(\lambda_1,\varphi_1)$ be the first eigenvalue and positive eigenfunction of $(-\Delta)^\alpha$ in $H^\alpha_0(\Omega)$, then for $k\in(0,k^*)$
\begin{eqnarray*}
  \int_\Omega u_k^p\varphi_1\,dx
&=&   \int_\Omega u_k(-\Delta)^\alpha\varphi_1\, dx- k\varphi_1(0)\\
&< &    \lambda_1 \left(\int_\Omega u_k^p\varphi_1\,dx\right)^{\frac1p} \left(\int_\Omega  \varphi_1\,dx\right)^{1-\frac1p}
\end{eqnarray*}
which implies that
$$\norm{u_k}_{L^p(\Omega,\rho^\alpha\,dx)}\le \lambda_1^{\frac{p}{p-1}}\int_\Omega  \varphi_1\,dx.$$
Combine the mapping $k\mapsto u_k$ is increasing, then $u_{k^*}=\lim_{k\nearrow k^*} u_k$ exists and $u_k\to u_{k^*}$ in $  L^p(\Omega,\rho^\alpha\,dx)$,
thus,
$$\int_{\Omega} u_{k^*}(-\Delta)^\alpha  \xi \,dx =\int_{\Omega} u_{k^*}^p\xi \,dx +{k^*}\xi(0), \quad \forall \xi\in C^\infty_c(\Omega).$$
So we conclude that (\ref{eq 1.1}) has a weak solution and then (\ref{eq 1.1}) has minimal solution $u_{k^*}$.

{\it To prove that $u_{k^*}$ is semi-stable.}
For any $\epsilon>0$ and $\xi\in H^\alpha_0(\Omega)\setminus\{0\}$, there exists $k(\epsilon)>0$ such that for all $k\in(k(\epsilon),k^*)$,
\begin{eqnarray*}
p\int_{\Omega} u_{k^*}^p\xi \,dx   \le   p\int_{\Omega} u_{k}^p\xi \,dx    +
  (k^*-k) p\int_{\Omega} u_{k^*}^p\xi \,dx \le   \norm{\xi}^2_{\alpha} +\epsilon
\end{eqnarray*}
By the arbitrary of $\epsilon$, we have that
$u_{k^*}$ is semi-stable.

{\it To prove the uniqueness. } If problem (\ref{eq 1.1}) admits a solution $u>u_{k^*}$.
$$
\sigma_1:= \inf_{\xi\in H^\alpha_0(\Omega)\setminus\{0\}}\frac{\norm{\xi}^2_{\alpha}}{p\int_{\Omega}  u_{k^*}^{p-1}\xi^2 \,dx}\ge 1.
$$
 By the compact embedding theorem, $\sigma_1$ is achievable and  its achieved function $\xi_1$ could be setting to be nonnegative and   satisfies
$$(-\Delta)^\alpha    \xi_1   =  \sigma_1 pu_{k}^{p-1}   \xi_1\ \ {\rm in}\ \ \Omega, \qquad \xi=0\ \ {\rm in}\ \ \R^N\setminus\Omega.$$
Letting $w=u-u_{k^* }>0$, then we have that
$$w=\mathbb{G}_\alpha[u_{k^*}^p- w^p].$$
By the elementary inequality
$$(a+b)^p> a^p+pa^{p-1}b \quad{\rm for}\ \ a>b>0,$$
we infers that
$$w> \mathbb{G}_\alpha[p u_{k^*}^{p-1} w],$$
then
\begin{eqnarray*}
\sigma_1\int_{\Omega}pu_{k^*}^{p-1} w\xi_1\,dx    =     \int_{\Omega}(-\Delta) w \xi_1\,dx
 >   \int_{\Omega}p u_{k^*}^{p-1} w\xi_1\,dx
\end{eqnarray*}
which is impossible with $\sigma_1\le 1$.  As a conclusion, $u_{k^*}$ is the unique solution of (\ref{eq 1.1}) with $k=k^*$.\hfill$\Box$

\subsection{Mountain-Pass type solution}

For the second solution of (\ref{eq 1.1}), we would like to apply the Mountain-Pass theorem to find a positive weak solution of
\begin{equation}\label{eq 4.1}
  \arraycolsep=1pt
\begin{array}{lll}
 \displaystyle   (-\Delta)^\alpha   u= (u_{k}+u_+)^{p}- u_{k}^p \quad
 &{\rm in}\quad \Omega,\\[2mm]
 \phantom{  (-\Delta)^\alpha}
 u =0 &{\rm in}\quad \R^N\setminus\Omega,
\end{array}
\end{equation}
where $k\in (0,k^*)$ and $u_{k}$ is the minimal positive solution of (\ref{eq 1.1}) obtained by Thoerem \ref{teo 0}.
The second solution of (\ref{eq 1.1}) is derived by following proposition.
\begin{proposition}\label{pr 4.1}
Assume that   $p\in(1,\frac N{N-2\alpha})$, $k\in (0,k^*)$   and
$u_{k}$ is the minimal positive solution of (\ref{eq 1.1}) obtained by Thoerem \ref{teo 0}.

Then problem (\ref{eq 4.1}) has a positive solution $v_k>u_{k}$.

\end{proposition}
\begin{proof}
We would like to employe the Mountain Pass theorem to look for the weak solution of (\ref{eq 4.1}).  A function $v$ is said to be a  weak solution of (\ref{eq 4.1}) if
\begin{equation}\label{weak 1}
\langle u,\xi\rangle_\alpha =\int_{\Omega}\left[ (u_{k}+u_+)^{p}- u_{k}^p\right]\xi \,dx,\quad \forall\xi\in  H^\alpha_0(\Omega).
\end{equation}
The natural functional associated to (\ref{eq 4.1}) is the following
\begin{equation}\label{E}
E(v)=\frac12\norm{v}_\alpha^2-\int_{\Omega}F(u_{k},v_+)\,dx, \quad\forall v\in  H^\alpha_0(\Omega),
\end{equation}
where
\begin{equation}\label{F}
 F(s,t)=\frac1{p+1}\left[(s+t_+)^{p+1}-s^{p+1}-(p+1)s^pt_+\right].
\end{equation}

We observe that for any $\epsilon>0$, there exists some $c_{\epsilon}>0$, depending only on $p$, such that
$$0\le  F(s,t) \le (p+\epsilon)s^{p-1}t^2+c_{\epsilon}t^{p+1},\quad s,t\ge0$$
By  we have that for any $v\in  H^\alpha_0(\Omega)$,
\begin{eqnarray*}
 \int_{\Omega}F(u_{k},v_+) \,dx  &\le& (p+\epsilon)\int_{\Omega}u_{k}^{p-1}v_+^2 \,dx +c_{\epsilon}\int_{\Omega}v_+^{p+1} \,dx  \\
    &\le& c_{22}\norm{v}_\alpha^2,
\end{eqnarray*}
then $E$ is well-defined in $  H^\alpha_0(\Omega)$.

We observe that $E(0)=0$ and let $v\in  H^\alpha_0(\Omega)$ with $\norm{v}_{ \alpha}=1$, then
 for  $k\in(0,k^*)$, choosing $\epsilon>0$ small enough, it infers from (\ref{e.3.1}) that
\begin{eqnarray*}
   E(tv) &= & \frac12t^2\norm{v}_\alpha^2- \int_{\Omega}F(u_{k},tv_+)\,\,dx \\
   &\ge&t^2\left(\frac12\|v\|_\alpha^2- (p+\epsilon) \int_{\Omega}  v_k^{p-1} v^2 \,dx\right)- c_{23} t^{p+1}\int_{\Omega}   |v|^{p+1} \,dx
   \\
   &\ge &c_{24} t^2\|v\|_\alpha^2 - c_{23} t^{p+1} \|v\|_\alpha^{p+1}
    \\
   &\ge & \frac{c_{24}}2t^2-c_{23}t^{p+1},
\end{eqnarray*}
where $c_{23},\, c_{24}>0$ depend on $k,k^*$ and we used   (\ref{3.1}) in the first inequality.
So there exists   $\sigma_0>0$  small, then  for $\|v\|_{ H^\alpha_0(\Omega)} =1$,  we have
$$ E(\sigma_0 v)\ge    \frac{c_{24}} 4\sigma_0^2=:\beta>0.$$

 We take a nonnegative function $v_0\in  H^\alpha_0(\Omega)$ and then
 $$ F(u_{k},tv_0)\ge \frac{1}{p+1}t^{p+1} v_0^{p+1}-tu_k^pv_0.$$
  Since the space of $\{tv_0: t\in\R\}$ is
a subspace of $ H^\alpha_0(\Omega)$ with dimension 1 and all the norms are equivalent, then $\int_{\Omega} V_0v_0(x)^{p+1} \,dx>0$.
Then there exists $t_0>0$ such that for $t\ge t_0$,
\begin{eqnarray*}
  E(tv_0)&=& \frac {t^2}2\|v_0\|_\alpha^2- \int_{\Omega}   F(u_{k},tv_0)  \,dx\\
   &\le &  \frac {t^2}2\|v_0\|_\alpha^2-c_{24}t^{p+1} \int_{\Omega}   v_0^{p+1} \,dx+t\int_{\Omega}   u_{k}^{p}v_0 \,dx
  \\ &\le &c_{25}(t^2+t-t^{p+1})\le  0,
\end{eqnarray*}
where $c_{24},c_{25}>0$. We choose $e=t_0v_0$, we have $E(e)\le0$.

We next prove that $E$ satisfies $(PS)$ condition. We say that $ E$ has $P.S.$ condition, if for any sequence
$\{v_n\}$ in $ H^\alpha_0(\Omega)$ satisfying $ E(v_n)\to c$ and $ E'(v_n)\to0$ as $n\to\infty$, there
is a convergent subsequence. Here the energy level $c$ of functional $E$ is characterized by
\begin{equation}\label{c}
 c=\inf_{\gamma\in \Upsilon}\max_{s\in[0,1]}E(\gamma(s)),
\end{equation}
where $\Upsilon =\{\gamma\in C([0,1]:\, H^\alpha_0(\Omega)):\gamma(0)=0,\ \gamma(1)=e\}$. We observe that
$$c\ge \beta.$$

Let $\{v_n\}$ in $ H^\alpha_0(\Omega)$ satisfying $ E(v_n)\to c$ and $ E'(v_n)\to0$ as $n\to\infty$, then we only have to show that there are a subsequence, still denote it by  $\{v_{n}\}$
 and $v\in  H^\alpha_0(\Omega)$ such that
$$v_{n}\to v\quad {\rm in}\ \ L^2(\Omega,  u_{k}^{p-1}\,dx)\quad{\rm and}\ \ L^{p+1}(\Omega) \quad {\rm as}\ n\to\infty.$$

For some $c_{21}>0$, we have that
$$
 c_{21}\|w\|_\alpha\ge E'(v_n)w
   =   \langle v_n, w\rangle_{\alpha} -\int_{\Omega}  f(u_{k}, (v_n)_+)w\,dx
$$
and
\begin{eqnarray}\label{4.2}
 c+1\ge E(v_n) = \frac12\|v_n\|^2_\alpha - \int_{\Omega}  F(u_{k}, (v_n)_+)\,dx.
\end{eqnarray}

Let $c_p=\min\{1,p-1\}$, then it follows by \cite[C.2 $(iv)$]{NS}  that
$$f(s,t)t-(2+c_p)F(s,t)\ge -\frac{c_pp}{2}s^{p-1}t^2,\quad s,t\ge0,$$
thus $(2+c_p)\times $(\ref{4.2})-$\langle E'(v_n),(v_n)_+\rangle$ implies that
\begin{eqnarray*}
c+ c_{21}\|v_n\|_\alpha&\ge&  \frac{c_p}2\|v_n\|^2_\alpha -\int_{\Omega} \left[(2+c_p)F(u_{k},(v_n)_+)-f(u_{k},(v_n)_+)(v_n)_+ \right]\,dx   \\
    &\ge &  \frac{c_p}2 \left[\|v_n\|^2_\alpha-p \int_{\Omega}  u_{k}^{p-1}v_n^2\,dx\right]
    \\&\ge&  c_{26} \frac{c_p}2\|v_n\|^2_\alpha,
\end{eqnarray*}
where $c_{26}>0$.
Therefore, we derive that $v_n$ is uniformly bounded in $ H^\alpha_0(\Omega)$ for $k\in(0,k^*)$.

Thus  there exists a subsequence $\{v_{n}\}$ and $v$ such that
$$v_{n}\rightharpoonup v\quad{\rm in}\quad  H^\alpha_0(\Omega), $$
$$v_{n}\to v\quad{\rm a.e.\ in}\ \Omega\quad{\rm and\ \ in}\quad L^{p+1}(\Omega),\ \ L^{2}(\Omega,  u_{k}^{p-1}\,dx),   $$
when $n\to \infty$. Here we have used that
$u_{k}^{p-1}\le c_{27}|x|^{(2\alpha-N)(p-1)},$
where $(2\alpha-N)(p-1)>-2\alpha$ and by Proposition \ref{pr 2.2} the embedding: $H^\alpha_0(\Omega)\hookrightarrow L^q(\Omega,  |x|^{(2\alpha-N)(p-1)}\,dx)$
is compact for $q\in[1, \frac{2N+2(2\alpha-N)(p-1)}{N-2\alpha})$,  particularly, for $q=2$.

We observe that
\begin{eqnarray*}
&&|F(u_{k},v_n)- F(u_{k},v)|
\\&&\qquad= \frac1{p+1}|(u_{k}+(v_n)_+)^p-(u_{k}+v_+)^p-(p+1)u_{k}^p((v_n)_+-v_+)| \\
   &&\qquad\le c_{28}u_{k}^{p-1} ((v_n)_+-v_+)^2+c_{28}((v_n)_+-v_+)^{p+1},
\end{eqnarray*}
which implies that
$$F(u_{k},v_n)\to F(u_{k},v)\quad{\rm a.e.\ in}\ \Omega\quad{\rm and\ \ in}\quad L^1(\Omega).$$
Then, together with $\lim_{n\to\infty} E(v_{n})=c$,
we have that $\|v_{n}\|_\alpha\to \|v\|_\alpha$ as $n\to\infty$.
Then we obtain that
$v_{n}\to v$ in $ H^\alpha_0(\Omega)$ as $n\to\infty$.

Now Mountain Pass Theorem  (for instance, \cite[Theorem~6.1]{struwe}; see also \cite{rabinowitz}) is applied to obtain that
there exists a critical point $v\in  H^\alpha_0(\Omega)$
 of $ E$ at some value $c\ge \beta>0$. By $\beta>0$, we have that $v$ is nontrivial and nonnegative. Then $v$ is a positive weak solution of $v$
of (\ref{eq 4.1}).  By using bootstrap argument in \cite{HL}, the interior regularity of $v$ could be improved to be in $ H^\alpha_0(\Omega)\cap C^2(\Omega\setminus\{0\})$, since $u_{k}$ is locally bounded in $\Omega\setminus\{0\}$ and $p<\frac{N}{N-2\alpha}$. Since
$$ u_{k}^{p-1}(x)\le c_{27}|x|^{(N-2\alpha)(p-1)},\quad \forall x\in B_1(0),$$
with $p<\frac{N}{N-2\alpha}$, we have that there is some $q>\frac N{2\alpha}$ such that
$$ u_{k}^{p-1}\in L^q(B_1(0)),$$
so $v_k$ is bounded at the origin.  Moreover, by Maximum Principle, we conclude that $v>0$ in  $\Omega$.
\end{proof}

\medskip
\noindent{\it Proof of the existence Mountain-Pass type solution in Theorem \ref{teo 1}.} From Proposition \ref{pr 4.1}, we obtain that there is a  positive weak solution of $v_k$
of (\ref{eq 4.1}), then $v_k$ is weak solution of (\ref{eq 4.1}) and it holds that
$$
\int_{\Omega} v_k(-\Delta)^\alpha  \xi \,dx = \int_{\Omega}    \left[(u_{k}+v_k)^{p}-u_{k}^p\right]\xi \,dx, \quad \forall \xi\in C^\infty_c(\Omega).
$$
then $(u_{k}+v_k)$ satisfies
$$
\int_{\Omega} (u_{k}+v_k)(-\Delta)^\alpha  \xi \,dx = \int_{\Omega}  (u_{k}+v_k)^{p}\xi \,dx +k\xi(0), \quad \forall \xi\in C^\infty_c(\Omega).
$$
This means that $v_k +u_{k}$ is weak solution of  (\ref{eq 1.2}) such that  $v_k+u_{k}> u_{k}$ and $v_k+u_{k}$
to $C^2$ locally in $\Omega\setminus\{0\}$.
\hfill$\Box$

\end{document}